\documentclass[a4paper,10pt]{article}

\usepackage{wrapfig}
\usepackage{graphicx}
\usepackage[utf8x]{inputenc}
\usepackage{amssymb}
\usepackage{enumerate}
\usepackage{amsmath}
\usepackage{amsthm}
\usepackage{mathtools}
\usepackage{lpic}
\usepackage[a4paper]{geometry}

\newcommand{\C}{\mathcal{C}}
\newcommand{\A}{\mathcal{A}}
\newcommand{\Z}{\mathbb{Z}}

\newcommand{\dd}{\; \mathrm{d}}

\newcommand{\R}{\mathbb{R}}
\newcommand{\ub}{\bar{u}}
\renewcommand{\sb}{\bar{s}}

\newcommand{\cb}{\bar{c}}

\newtheorem{Theorem}{Theorem}[section]
\newtheorem{Proposition}[Theorem]{Proposition}
\newtheorem{Lemma}[Theorem]{Lemma}
\theoremstyle{remark}
\newtheorem{Remark}[Theorem]{Remark}
\theoremstyle{definition}

\title{Traveling surface waves of moderate amplitude\\ in shallow water}
\author{Armengol Gasull, Anna Geyer}
\date{}
\begin{document}
 \maketitle
 

\begin{abstract}
\noindent We study traveling wave solutions of an equation for surface waves of moderate amplitude arising as a shallow water approximation of the Euler equations for inviscid, incompressible and homogenous fluids. We obtain solitary waves of elevation and depression, including a family of solitary waves with compact support, where the amplitude may increase or decrease with respect to the wave speed. Our approach is based on techniques from dynamical systems and relies on a reformulation of the evolution equation as an autonomous Hamiltonian system which facilitates an explicit expression for bounded orbits in the phase plane to establish existence of the corresponding periodic and solitary traveling wave solutions.

\end{abstract}

\section{Introduction and main result}
A number of competing nonlinear model equations for water waves have been proposed to this day to account for fascinating phenomena, such as wave breaking or solitary waves, which are not captured by linear theory. The well-known Camassa--Holm equation \cite{Camassa1993} is one of the most prominent examples, due to its rich structural properties. It is an integrable infinite-dimensional Hamiltonian system \cite{ Teschl2009, Con01, Constantin2006c} whose solitary waves are solitons \cite{Constantin2002, DikaMoli2007}. Some of its classical solutions develop singularities in finite time in the form of wave breaking \cite{Constantin1998a}, and recover in the sense of global weak solutions after blow up \cite{Bressan2007a,Bressan2007}. For a classification of its weak traveling wave solutions we refer to \cite{Lenells2005a}. The manifold of its enticing features led Johnson to demonstrate the relevance of the Camassa--Holm equation as a model for the propagation of shallow water waves of moderate amplitude. He proved that the horizontal component of the fluid velocity field at a certain depth 
within the fluid is indeed described by a Camassa--Holm equation \cite{Con11,Joh02}. Constantin and Lannes \cite{ConLan09} followed up on the matter in search of a suitable corresponding equation for the free surface and derived an evolution equation for surface waves of moderate amplitude in the shallow water regime,
\begin{align}
  \label{MAE}
  u_t + u_x + 6u u_x - 6u^2u_x + 12u^3u_x+
      u_{xxx} -u_{xxt}  + 14uu_{xxx} + 28u_xu_{xx} = 0,
\end{align}
The authors show that equation \eqref{MAE}  approximates the governing equations to the same order as the Ca\-mas\-sa--Holm equation, and also prove that the Cauchy problem on the line associated to \eqref{MAE}, is locally well-posed \cite{ConLan09}. Employing a semigroup approach due to Kato \cite{Kato1975}, Duruk \cite{DurukMutlubas2013a} shows that this results also holds true for a larger class of initial data, as well as for the corresponding spatially periodic Cauchy problem \cite{DurukMutlubas2013b}. Consequently, solutions of \eqref{MAE} depend continuously on their initial data in $H^s$ for $s>3/2$, and it can be shown that this  dependence is not uniformly continuous \cite{DGM13}. In the context of Besov spaces,  well-posedness is discussed \cite{Mi2013} using Littlewood-Paley decomposition, along with a study about analytic solutions and persistence properties of strong solutions. One of the important aspects of equation \eqref{MAE} lies in its relevance for capturing the non-linear phenomenon of wave breaking \cite{ConLan09,DurukMutlubas2013a}, a feature it shares with the Camassa-Holm equation. While the latter equation is known to possess global solutions \cite{Bressan2007,Constantin1998}, it is not apparent how to obtain global control of the solutions of equation \eqref{MAE}, owing to its involved structure and due to the higher order nonlinearities. However, passing to a moving frame one can study so-called traveling wave solutions, whose wave profiles move at constant speed in one direction without altering their shape. Existence of solitary traveling waves which decay to zero at infinity has been established \cite{Gey12c} for wave speeds $c>1$, and their orbital stability has been deduced \cite{DurukGeyer2013a} employing an approach due to Grillakis, Shatah and Strauss \cite{Grillakis1987} taking advantage of the Hamiltonian structure of \eqref{MAE}.\\

\noindent In the present paper we set out to improve the existence result \cite{Gey12c}  by loosening the assumption that solitary waves tend to zero at infinity. Allowing for a decay to an arbitrary constant, we establish existence of a variety of novel traveling wave solutions of \eqref{MAE}. 
\begin{Theorem}
  \label{thm}
For every speed $c\in \R\backslash\{c^*\}$ there exist peaked periodic, as well as  smooth solitary and periodic traveling wave solutions of \eqref{MAE}. Periodic waves may be obtained also for $c^*\in\R$, whose value is given in \ref{A:1}. \\

\noindent Moreover, the solitary waves can be characterized in terms of two parameters -- the wave speed $c$ and the level of the undisturbed water surface $s$ -- allowing us to determine the exact regions in this parameter space which give rise to the following types of waves, cf.~Figure \ref{Fig:amp}:
 \begin{itemize}
  \item Solitary waves with \emph{compact support} (along the straight line given by $ c+1+14\,s=0$).
  \item Solitary waves of \emph{elevation} whose amplitude is strictly \emph{increasing} (in region I) or \emph{decreasing} (in regions II and III) with respect to $c$.
  \item Solitary waves of \emph{depression} whose amplitude is strictly \emph{increasing} (in region V) or \emph{decreasing} (in region IV) with respect to $c$.
  \end{itemize}
 All solitary waves are symmetric with respect to their unique crest/trough, they are monotonic and decay exponentially to the undisturbed water surface $s$ at infinity.
  \end{Theorem}

\begin{figure}[h]
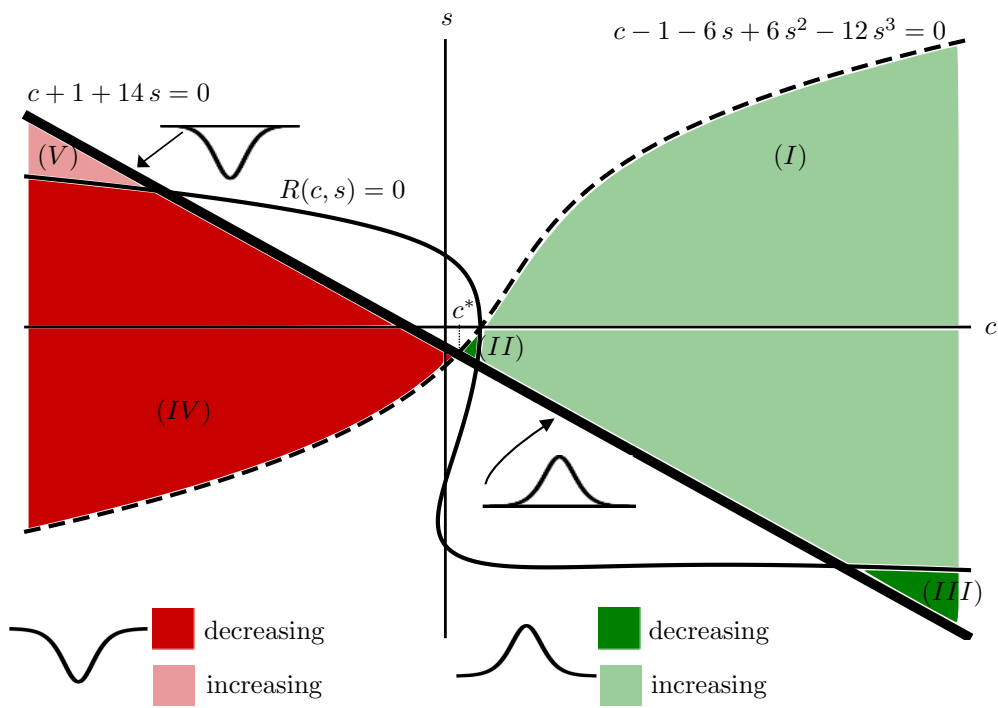

  \centering
    \begin{lpic}[l(0mm),r(0mm),t(0mm),b(0mm)]{regionsampcp(0.65)}
      \lbl[l]{155,112;$(I)$}
      \lbl[l]{185,23;$(III)$}
      \lbl[l]{95,73;$(II)$}
      \lbl[l]{90,81;$c^*$}
      \lbl[l]{30,60;$(IV)$}
      \lbl[l]{6,112;$(V)$}

      \lbl[m]{199,77;$c$}
      \lbl[m]{89,140;$s$}
      \lbl[l]{4,125;$c+1+14\,s=0$} 
      \lbl[r]{190,138;$c-1-6\,s+6\,s^2-12\,s^3=0$}

      \lbl[l]{55,105;$R(c,s)=0$}  
      
      \lbl[m]{52,15;decreasing}
      \lbl[m]{52,4; increasing}      
      
      \lbl[m]{142,15;decreasing}
      \lbl[m]{142,4; increasing}  
        
    \end{lpic}
   \caption{We obtain a variety of traveling waves for \eqref{MAE} with different behaviour, cf.~Thm \ref{thm}: solitary waves of \emph{elevation} and \emph{depression} which \emph{increase} or \emph{decrease} with $c$. Along the straight line $c+1+14\,s=0$, we obtain solitary waves with \emph{compact support}. The algebraic curve $R(c,s)=0$ arising from a polynomial of degree nine is given in Section \ref{subsect:ampl} below.
 }
 \label{Fig:amp}
 \end{figure}
 
\noindent The proof of these results hinges on the observation that for traveling waves, equation \eqref{MAE} may be written as an autonomous Hamiltonian system involving two parameters. This insight allows us to explicitly determine bounded orbits in the phase plane which correspond to solitary and periodic traveling waves of elevation as well as depression (Section \ref{subsect:Hamiltonian} and  \ref{subsect:CondExistence}). Moreover, we characterize all solitary traveling waves in terms of two parameters --  the wave speed $c$ and the water level $s$ of the undisturbed surface at infinity. This enables us to prove the existence a family of solitary waves with compact support (Section \ref{Section:finitehomorbit}). Furthermore, we obtain a family of peaked periodic waves (Section \ref{subsect:peaked}). Our work also  extends the analysis of qualitative properties regarding the shape of solitary waves given in \cite{Gey12c}: we prove that the profile is strictly monotonic between crest and trough, and derive explicit algebraic curves in the parameter space $(c,s)$ to determine the regions where their amplitude  is increasing and  decreasing with respect to the wave speed $c$ (Section \ref{sect:Props}). These are quite remarkable properties of solitary waves which, to our knowledge, contribute novel aspects to the study of traveling waves in evolution equations for water waves. Our approach is in fact applicable to a large class of nonlinear dispersive equations, which we exemplify by a discussion of traveling waves of the aforementioned Camassa--Holm equation, including peaked continuous solitary waves (Section \ref{sect_CH}). Some of the more involved computations regarding the algebraic curves are carried out in Appendix \ref{A:}.
%
%
 
 \newpage
 \section{Existence of traveling waves}
 \label{sect:Existence}
The proof of Theorem \ref{thm} relies on the fact that equation \eqref{MAE} has very nice structural features.
\subsection{Hamiltonian Formulation}
\label{subsect:Hamiltonian}
 Consider a general partial differential equation with constant coefficients in $u(t,x)$ for $(t,x)\in \R^2$, which upon introducing the traveling wave Ansatz 
\begin{equation*}
 \label{cov}
  \xi = x-c\,t, \quad u(\xi)=u(t,x),
\end{equation*}
can be transformed into an autonomous ordinary differential equation of the form
  \begin{equation}
      \label{MAEODE}
      \ddot u\,(u-\ub) + \frac{1}{2}(\dot u)^2 + F'(u)= 0,
  \end{equation} 
 where $\ub$ is a constant, $F(u)$ is a smooth function and  the dot denotes differentiation with respect to $\xi$. The corresponding planar system is given by
  \begin{equation}
  \label{Sys}
       \left\{
      \begin{array}{l l }
	\dot u =v\vspace{0.8em}\\
	\dot v = \dfrac{- F'(u) - \frac{1}{2}\,v^2}{u-\ub},
      \end{array}\right.
  \end{equation}
and we observe that  a reparametrisation of the independent variable according to $\frac{d\xi}{d\tau}=u-\ub$ transforms \eqref{Sys} into
\begin{align}
  \label{HSys} 
      \left\{
      \begin{array}{l l }
	u' = (u-\ub)\,v \vspace{0.8em}\\
	v' = - F'(u) - \frac{1}{2}\,v^2,\\
      \end{array}\right.
  \end{align} 
where the prime denotes differentiation with respect to  $\tau$. The latter system  is clearly topologically equivalent to the former (cf.~\cite{Dumortier2009, Guggenheimer1983})  on each connected component of $\R \backslash \{u=\ub\}$, preserving orientation in the open half-plane $\{u>\ub\}$ and reversing orientation in the other half. 
The advantage of \eqref{HSys} is that it possesses a Hamiltonian
  \begin{equation}
   \label{Hamiltonian}
     H(u,v) = F(u) + \frac{1}{2}\,v^2\,(u-\ub) =h
  \end{equation}
satisfying $u'=H_v$ and $v'=-H_u$, which is constant along the solution curves of \eqref{HSys}. Explicit knowledge of the critical points and (closed) orbits 
  \begin{equation}
    \label{v}
      v=\pm \sqrt{2\,\dfrac{h-F(u)}{u-\ub}}
  \end{equation}
  in the phase plane of \eqref{HSys} therefore completely characterizes the smooth traveling wave solutions of the partial differential equation. Applying these ideas to \eqref{MAE}, we find that the associated  equation for traveling waves
may be written in the form \eqref{MAEODE} with
 \begin{equation}
    \label{FMAE}
	F(u)=  K\,u + \frac{1-c}{28}u^2 + \frac{1}{14}u^3 - \frac{1}{28}u^4 + \frac{3}{70}u^5,
 \end{equation}
 where $K$ is a constant of integration and
 \[
	\ub = -\frac{1+c}{14}.
 \]
In view of the above considerations, we obtain the following existence result for bounded traveling wave solutions of \eqref{MAE}:
\begin{Proposition}
\label{Prop:orbits_waves}
Solitary wave solutions of \eqref{MAEODE} are obtained from the homoclinic connection based at the saddle point of \eqref{HSys}, whereas periodic waves correspond to periodic orbits around the center. These solutions are symmetric with respect to the crest/trough and have one maximum per period. The solitary waves tend  exponentially to a constant on either side of the crest/trough.
\end{Proposition}

\begin{proof}
In order to obtain bounded orbits of system \eqref{HSys}, we study the critical points of the Hamiltonian system which are closely related to the local extrema of $F$, since $(u',v') = (0,0)$ exactly when $v=0$ and $F'(u)=0$. Notice that $u=\ub$ is an invariant line for \eqref{HSys}. The fact that any non-degenerate critical point of an analytic Hamiltonian system is either a topological saddle or a center (cf.~\cite{Perko2006}, p.\,154) simplifies our analysis considerably. Computing the Jacobian $J$ of \eqref{HSys} and evaluating it at the critical points shows that $\det J=F''(u)\,(u-\ub)$. Recall that a non-degenerate critical point is a center whenever $\det J >0$ and a topological saddle when $\det J<0$, cf.~\cite{Perko2006}. Hence, all further analysis regarding the number, location and type of critical points in the phase plane is based on the specific structure of the polynomial $F$, depending on the parameters $c$ and $K$.  It is straightforward to check that $F$ has at most two local extrema
and we conclude that system \eqref{HSys} has at most two critical points. \\

\noindent Next we show how to obtain the expressions for bounded orbits corresponding to bounded traveling wave solutions from relation \eqref{v}, cf. Figure \ref{orbits}, and infer some basic properties of the waves. Homoclinic orbits are obtained by letting $h=h_s$, where $h_s=F(s)$ and  $s$ solves
\begin{align}
\label{homoclinicorbits}
	&F'(y)=0 \quad \text{ and }\quad 
    \left\{
    \begin{array}{ll}
	    F''(y) < 0 \;\text{ when } y > \ub,\\
		F''(y) > 0 \;\text{ when } y < \ub.
    \end{array}\right.
\end{align} 
We will see below that this definition of $s$ giving rise to  the saddle point $(s,0)$ of system \eqref{HSys}, aptly captures the physical interpretation of $s$ as the level of  the undisturbed water surface at infinity  in certain regions of the parameter space. \\

\noindent Regarding $v$ in \eqref{v} as a function of $u$ and choosing $h=h_s$ yields an explicit expression of the homoclinic connection in the phase plane. The orbit leaves the critical saddle point $(s,0)$ and crosses the horizontal axis once at $(m,0)$, before returning to the saddle point symmetrically with respect to the horizontal axis. The value $m$ is obtained at the unique intersection of the horizontal line $h_s$ with the polynomial $F$, where $F(m)=F(s)$. The corresponding solitary wave solution therefore has a unique maximum $m$ and is symmetric with respect to the crest point. 
\begin{figure}[!ht]
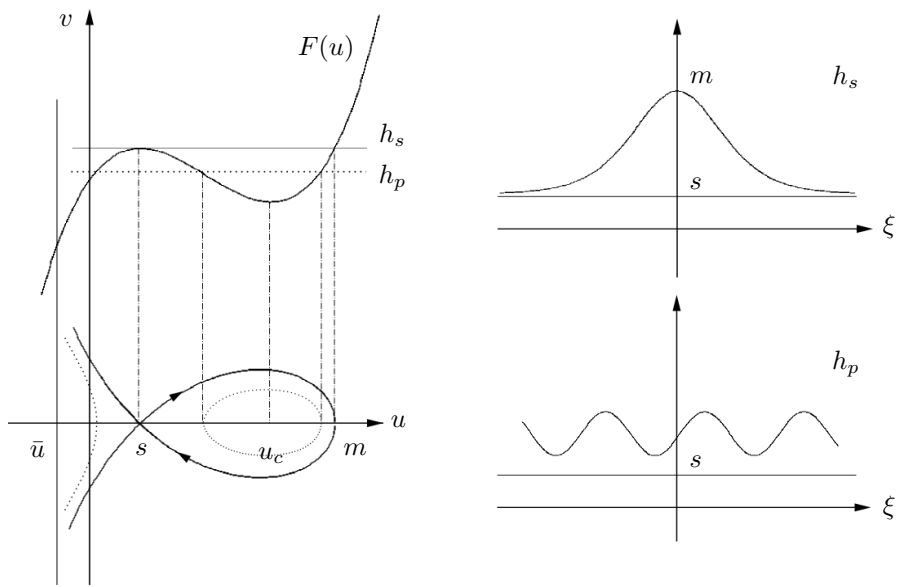

 \centering
 \begin{lpic}[l(5mm),r(15mm),t(0mm),b(0mm)]{orbits_waves_nonames(0.38)}
      \lbl[m]{110,190;$F(u)$}
      \lbl[m]{133,160;$h_s$}
      \lbl[m]{133,145;$h_p$}
      \lbl[m]{135,60;$u$}
      \lbl[m]{20,200;$v$}
      \lbl[m]{10,50;$\ub$}
      \lbl[m]{120,50;$m$}
      \lbl[m]{91,48;$u_c$}
      \lbl[m]{46,50;$s$}
      \lbl[m]{290,180;$h_s$}
      \lbl[m]{240,180;$m$}
      \lbl[m]{238,143;$s$}
      \lbl[m]{290,80;$h_p$}
      \lbl[m]{238,46;$s$}
      \lbl[m]{305,128;$\xi$}
      \lbl[m]{305,30;$\xi$}
  \end{lpic}
   \caption{A sketch of how bounded orbits in the $(u,v)$ phase plane are obtained using relation \eqref{v}, where choosing $h=h_p$ and $h=h_s$ give rise to periodic and solitary traveling waves, respectively. }
     \label{orbits}
\end{figure}
Furthermore, the solution decays exponentially to the constant $s$ on either side of the crest. Indeed, by the Hartman-Grobman Theorem (cf.~\cite{Teschl2011a,Sotomayor1979}) the vectorfield $(\dot u,\dot v)$ is locally $\C^1$-conjugate to its linearization at the hyperbolic saddle point $(s,0)$, and therefore solutions on the stable manifold converge exponentially to the fixed point. The decay rate is given by the eigenvalues of the Jacobian at the saddle point of system \eqref{Sys},  hence we recover the result for $K=0$ obtained in \cite{Gey12c}.  Observe that for $s<\ub$ we obtain solitary waves of depression with the same qualitative properties. \\

\noindent Similarly to solitary waves, periodic traveling waves are obtained by choosing $h\in(h_c,h_s)$ in \eqref{v} where $h_c=F(u_c)$  and $u_c$ is a solution  of 
 \begin{align*}
	&F'(y)=0 \quad \text{ and } \quad
    \left\{
    \begin{array}{ll}
	    F''(y) > 0 \;\text{ when } y > \ub,\\
		F''(y) < 0 \;\text{ when } y < \ub.
    \end{array}\right.
 \end{align*} 
These periodic waves undulate about $u=u_c$.
\end{proof}

\subsection{Conditions for the existence of solitary traveling waves}
\label{subsect:CondExistence}
We now derive algebraic conditions for the existence of homoclinic orbits, which give rise to solitary  traveling wave solutions of \eqref{MAE} as we have just seen. 
At this point, our problem involves  three interdependent parameters: 
\begin{center}
 \begin{tabular}{c c l}
    $c$ & $\dots$ & the wave speed, \\
    $K$ & $\dots$ & the constant of integration, \\
    $s$ & $\dots$ & the level of the undisturbed water surface at infinity.
 \end{tabular}
\end{center}
It turns out to be more convenient to eliminate the parameter $K$ in favor of $s$ using relation $F'(s)=0$ from the  definition \eqref{homoclinicorbits}  of $s$. This leads us to the following change of parameters:

\begin{equation}
\label{Kcs}
  \Phi: (c,s) \longmapsto (c,K) =(c,\varphi(c,s)),
\end{equation}
where 
\begin{equation}
\label{phi}
 \varphi(c,s) =\frac{1}{14}\,s \,( -3\,{s}^{3}+2\,{s}^{2}-3\,s+c -1).
\end{equation}
\begin{Remark}
\label{R:Ks}
Notice that $\Phi$ is not bijective on $\R^2$ . For instance, for each $c\in\R$ and  $K$ big enough the point $(c,K)$ has no preimage. However, this happens precisely when there are no simple solutions of $F'(s)=0$, in which case system \eqref{HSys} does not have homoclinic orbits. In the remaining cases, observe that for each fixed $c\in\R$ there exist two $s_1 \neq s_2$ such that $\varphi(c,s_1)=\varphi(c,s_2)=(c,K)$. This leads to a redundancy in the parameter regions, since for each $c$ there are two values $s$ which yield the same $K$, and hence the same $F$,  which gives rise to the same phase portrait of system \eqref{HSys}, and therefore also to the same wave solution. The difference between the two values of $s$ is that one of them, say $s_1$, satisfies \eqref{homoclinicorbits} and hence corresponds to the saddle point of system \eqref{HSys}, thus reconciling  the interpretation of $s$ being the undisturbed water level at infinity. The other value, $s_2$, which satisfies the reverse inequalities, gives rise to a center of the system. The solitary wave solution corresponding to the homoclinic orbit around this center $(s_2,0)$ decays to the undisturbed water level given precisely by the former value $s_1$, which means that this solitary wave solution is already obtained via the value $s_1$. To avoid this redundancy, we select  the value $s$ in \eqref{phi} satisfying
\begin{align}
\label{condF''}
      \left\{
    \begin{array}{ll}
	    F''(s) < 0 \;\text{ when } s > \ub,\\
		F''(s) > 0 \;\text{ when } s < \ub,
    \end{array}\right.
\end{align}
which  makes the transformation \eqref{Kcs} bijective on the relevant regions. In this way, we choose the value corresponding to the saddle point of the system, which is in accordance with the physical interpretation of $s$ being the level of the undisturbed water surface at infinity.
\end{Remark}

\noindent Performing the substitution $K=\varphi(c,s)$ given in  \eqref{phi} facilitates our analysis considerably and we proceed to study conditions for the existence of homoclinic orbits in terms of the parameters $s$ and $c$. 
The discussion in the proof of Proposition \ref{Prop:orbits_waves} ensures that \eqref{HSys} has at most two critical points corresponding to the local extrema of  $F$.  It is easy to see that homoclinic orbits exist when there is at most one saddle point and one center point in the phase plane, cf.~Figure \ref{orbits}. This situation occurs when  
\begin{enumerate}[(i)]
    \item $F$ has two distinct local extrema, and moreover,
    \item both extrema lie either to the left or to the right of the invariant line $u=\ub$. 
\end{enumerate}
To ensure condition (i) we study the roots of the discriminant of $F'$, which in view of the substitution \eqref{phi} is given by 
\begin{equation}
\label{boundADis}
    \text{Dis}(F'(u),u)=\alpha F''(s)^2\, M(c,s),
\end{equation}
where $\alpha<0$ is a real constant and $M(c,s)$ is a polynomial with no real roots, cf.~Appendix \ref{A:1}. Condition (ii) holds whenever $F'(\ub)$ is non-zero, i.e.~we study the roots of 
\begin{equation}
\label{boundAdF}
    F'(\ub)=\beta (s-\ub)\, N(c,s),
\end{equation}
where $\beta>0$ is a real constant and $N(c,s)$ is a cubic polynomial in $c$ and $s$. The algebraic curves corresponding to the zeros of these factors intersect precisely in one point $(c^*,s^*)$ in the parameter plane, cf.~Appendix \ref{A:1} and Figure \ref{MAEregions}. We distinguish the following six regions:
\begin{align}
 \begin{tabular}{c | c}
         $ s>\ub$                                 &       $s<\ub$                           \\ \hline
    $R_1$: $F''(s)<0$ and $N(c,s)>0$ & $R_4$:  $F''(s)>0$ and $N(c,s)<0$ \\
    $R_2$: $F''(s)>0$ and $N(c,s)>0$ & $R_5$:  $F''(s)<0$ and $N(c,s)<0$ \\
    $R_3$: $F''(s)>0$ and $N(c,s)<0$ & $R_6$:   $F''(s)<0$ and $N(c,s)>0$ 
 \end{tabular}
  \label{regions}
\end{align}

\begin{figure}[h]
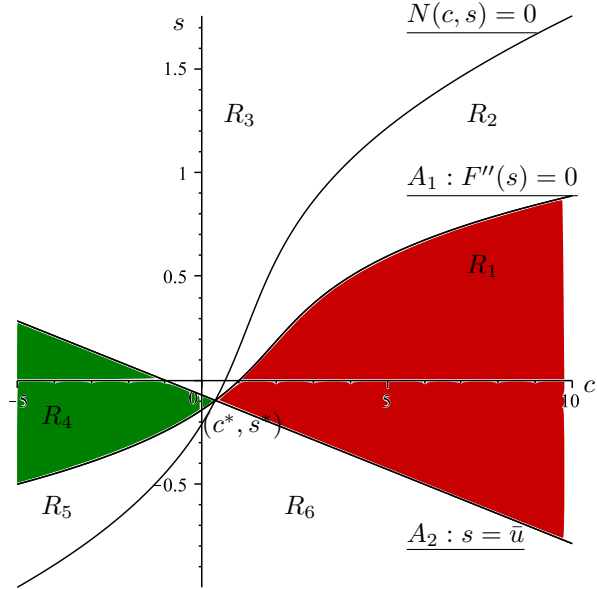

 \centering
  \centering    
 \begin{lpic}[l(0mm),r(5mm),t(-2mm),b(-3mm)]{regions_color(0.4)}
      \lbl[m]{160,110;$R_1$}
      \lbl[m]{160,160;$R_2$}
      \lbl[m]{80,160;$R_3$}
      \lbl[m]{20,60;$R_4$}
      \lbl[m]{20,30;$R_5$}
      \lbl[m]{100,30;$R_6$}

      \lbl[m]{81,57;$(c^*,s^*)$}      
      \lbl[m]{60,190;$s$}
      \lbl[m]{195,70;$c$}
      
      \lbl[l]{135,138;$\underline{A_1: F''(s)=0}$}
      \lbl[l]{135,20;$\underline{A_2: s=\ub}$}
      \lbl[l]{135,192;$\underline{N(c,s)=0}$}
  \end{lpic}
   \caption{The shaded region $\A=R_1\cup R_4$ in the parameter plane $(c,s)$ yields homoclinic orbits which give rise to solitary wave solutions of \eqref{MAEODE} traveling at speed $c$ and decaying to the undisturbed water level $s$ at infinity. When $s>\ub$ we obtain solitary waves of elevation $(R_1)$, whereas for $s<\ub$ we obtain solitary waves of depression  $(R_4)$. }
     \label{MAEregions}
\end{figure}
\noindent Let us focus first on the regions where $s>\ub$. Choosing $(c,s)$ in $R_3$, we have that $F'(\ub)<0$ in view of \eqref{boundAdF} and \eqref{regions}, meaning that there is one extremum of $F$ on each side of the invariant line $\ub$. In this case, both extrema of $F$ yield a center for system \eqref{HSys} which impedes the existence of a homoclinic connection. Hence, region $R_3$ gives rise to periodic orbits only. Choosing $(c,s)$ in $R_2$ violates the condition on $F''(s)$ in \eqref{condF''}, and hence we refrain from any further analysis (recall the redundancy discussed in Remark \ref{R:Ks}). Finally, $(c,s)$ in $R_1$ yields $F$ with two distinct local extrema to the right of $\ub$ since $F'(\ub)>0$, and hence $F$ has a local maximum in $s$ in view of $F''(s)<0$. The homoclinic orbit based in the corresponding saddle point $(s,0)$ gives rise to a solitary traveling wave solution of \eqref{MAE} which propagates at speed $c$ and decays to the undisturbed water level $s$ at infinity. Notice that all solitary wave solutions obtained in this way from parameters in region $R_1$ are waves of elevation. Applying the same reasoning to the regions where $s<\ub$, we find that $R_6$ gives rise to periodic orbits only, whereas choosing parameters in region $R_4$ yield solitary waves of depression. We conclude that there are two algebraic curves $A_1$ and $A_2$ bounding the region $\A=R_1 \cup R_4$ which admits solitary traveling wave solutions of \eqref{MAE}:
\begin{align}
 \left\{
     \begin{array}{ll}
        A_1 := \{ F''(s)  =-1-6\,s+6\,{s}^{2}-12\,{s}^{3}+c=0\}  \vspace{1em} \\
        A_2 := \{s-\ub = s+ \dfrac{1+c}{14}=0\}.
     \end{array}
 \right.
\label{boundsA}
\end{align}
We summarize our conclusions in the following Proposition. 
 \begin{Proposition}
 \label{Lem:Conditions_orbits}
 Solitary wave solutions of \eqref{MAE} propagating at speed $c$ and decaying to the undisturbed water level $s$ at infinity exist if and only if $(c,s) \in \A=R_1 \cup R_4$, the region which is bounded by the algebraic curves $A_1$ and $A_2$ defined in \eqref{boundsA} above. Parameters in $R_1$ yield solitary waves of elevation, whereas $R_4$ gives rise to solitary waves of depression.
 \end{Proposition}
 
\subsection{Solitary waves with compact support}
 \label{Section:finitehomorbit}
 It turns out that equation \eqref{MAE} admits smooth solitary wave solutions with compact support on $\R$. This is essentially due to the fact that the planar system  \eqref{Sys} is discontinuous along the straight line $u=\ub$ (for a detailed account on various evolution equations arising in the context of nonlinear water waves which yield so-called ''singular nonlinear travelling wave systems`` we refer to \cite{Li2013}). The intuition behind this surprising observation is that the homoclinic orbit corresponding to the compactly supported solitary wave has finite existence time when the local maximum of $F$ lies at the invariant line $\bar{u}$,  cf.~Figure \ref{Fig_homorbit}. This situation occurs precisely on the curve $A_2$, where $s=\ub$, in which case the level line of the Hamiltonian corresponding to the homoclinic orbit based in the saddle point is $h_s=F(\ub)$ and $F'(\ub)=0$. Therefore, relation \eqref{v} simplifies to 
 \begin{align}
  \label{vkl}
    v = \pm \sqrt{2\, \dfrac{ F(\ub) - F(u)}{u-\ub}} = \pm \sqrt{(u-\ub)\;p\,(u)},
 \end{align}
 where
 \begin{equation*}
 p\,(u):=F''(\ub) + \frac{2}{3!}F^{(3)}(\ub)(u-\ub) + \dots + \frac{2}{5!}F^{(5)}(\ub)(u-\ub)^3.
\end{equation*}
\begin{figure}[h]
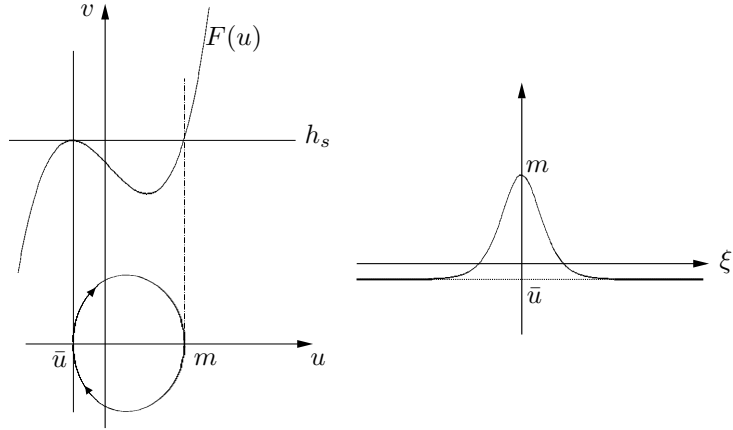

  \centering    
 \begin{lpic}[l(0mm),r(5mm),t(0mm),b(0mm)]{cpsolwave_noname(0.37)}
      \lbl[m]{100,140;$F(u)$}
      \lbl[m]{130,105;$h_s$}
      \lbl[m]{130,25;$u$}
      \lbl[m]{38,25;$\ub$}
      \lbl[m]{90,25;$m$}
      \lbl[m]{207,48;$\ub$}
      \lbl[m]{208,94;$m$}
      \lbl[m]{48,150;$v$}
      \lbl[m]{275,60;$\xi$}
  \end{lpic}
 \caption{The choice of parameters $c\in\R$ and $s=\ub$ in $F$ yields a homoclinic orbit with finite existence time, which gives rise to a solitary wave solution with compact support.}
     \label{Fig_homorbit}
\end{figure}

\noindent In particular, the existence time of these homoclinic orbits is finite. Indeed, notice that
\[
 T(u,u_0):= \int_{u_0}^{u} \frac{\dd r}{\sqrt{(r-\ub)\,p(r)}}\,
\]
is an elliptic integral and therefore finite, since $p(r)$ is a third degree polynomial with no repeated roots and $\ub$ is not a root of $p(r)$. In view of  \eqref{vkl} this yields 
\[
  T(u(\xi),u_0)=\int_{u(\xi_0)}^{u(\xi)} \frac{\dd r}{\sqrt{(r-\ub)p(r)}}= \int_{\xi_0}^{\xi} \frac{\sqrt{(u-\ub)p(u)}}{\sqrt{(u-\ub)p(u)}}\dd \xi =\xi-\xi_0,
\]
for a solution of $\dot u(\xi)=\sqrt{(u-\ub)p(u)}$ with initial data $u(\xi_0)=u_0$. Hence, the time it takes an orbit to get from $\ub$ to $m$, where $m$ is the non-trivial solution of $F(\ub)=F(m)$,  is given by
\[
  T := T(u(\xi),\ub) - T(u(\xi),m)= \int_{\ub}^{m}{\dfrac{\dd r}{ \sqrt{(r-\ub)\,p(r)}}} <\infty.
\]
By symmetry it follows that the solitary traveling wave solution corresponding to the orbit with $h_s=F(\ub)$ is defined on the finite interval $(-T,T)$. We may extend this solution to the real line by setting $u(\xi)=\ub$ for $\xi \in \R\backslash(-T,T)$. This is possible since $u=\ub$ is a constant solution of \eqref{MAEODE} when  $s=\ub$. Furthermore, $u(\xi) \rightarrow \ub$ as $\xi \rightarrow \pm T$ and therefore $v \rightarrow 0$ and $\dot v\rightarrow 0$ in view of \eqref{Sys}, applying de L'Hopital. This proves that the extension to $\R$ is $\C^2$. When $\xi$ approaches the finite existence time $\pm T$, the solution decays like 
\begin{equation}
    \label{decay}
    u(\xi) = \ub -\frac{1}{4} F''(\ub)(\xi\pm T)^2 + O((\xi\pm T)^3)
\end{equation}
which is readily checked.

\subsection{Peaked periodic waves}
\label{subsect:peaked}
For parameters in the regions $R_3$ and $R_6$, cf.~Figure \ref{MAEregions}, the invariant line $\ub$ lies between the two critical points of the polynomial $F$. In this case, the extrema of $F$ yield two centers in the phase portrait of \eqref{HSys} which impedes the existence of solitary waves (and hence, the parameter $s$ no longer accommodates the physical interpretation of the undisturbed water level at infinity). However, we show by continuous extension that there exist peaked periodic waves above and below the line $\ub$, and periodic waves undulating about $\ub$. \\

\noindent Indeed, for every $(c,s)\in R_3\cup R_6$, periodic waves are obtained as in Proposition \ref{Prop:orbits_waves} by choosing $h_p\in (h_1,h_2)$, where $h_i=F(u_i)$ for $i=1,2$, and $u_i$ is a solution of 
\begin{align*}
 \label{periodicorbits2}
	&F'(y)=0 \quad \text{ and } \quad
    \left\{
    \begin{array}{ll}
	    F''(y) > 0 \;\text{ when } y > \ub,\\
		F''(y) < 0 \;\text{ when } y < \ub,
    \end{array}\right.
 \end{align*} 
and employing \eqref{v}. We will now treat the special case $h_p=F(\ub)$.  Notice that, by construction,
\begin{align*}
    h_p- F(u) = (u - \ub)(u-m_1)(u-m_2)\,q(u),
\end{align*}
where $q(u)$ is a second order polynomial with no real roots and $m_i \neq \ub, i=1,2,$  are the other two intersections of the horizontal line $h_p$ with $F(u)$. Using \eqref{v}, we obtain two heteroclinic orbits of the system \eqref{HSys} leaving and returning to the invariant line $u=\ub$ given in terms of
\begin{equation}
 \label{vi}
  v_i = \pm \sqrt{2\,(u-m_1)(u-m_2)\,q(u)}, 
\end{equation}
for  $u\in(m_1,\ub)$ and $u\in(\ub,m_2)$ respectively, which intersect the horizontal axis at $m_1$ and $m_2$, where $m_1 < \ub < m_2$. Observe that for the topologically equivalent system \eqref{Sys}, the existence times of these orbits are again finite
and given in terms of 
\begin{equation*}
  T_1= \int_{m_1}^{\ub}{\dfrac{\dd u}{\sqrt{2\,(u-m_1)(u-m_2)\,q(u)}}} < \infty \;\text{ 
and }\;  T_2= \int_{\ub}^{m_2}{\dfrac{\dd u}{\sqrt{2\,(u-m_1)(u-m_2)\,q(u)}}} < \infty.
\end{equation*}
\begin{figure}[!ht]
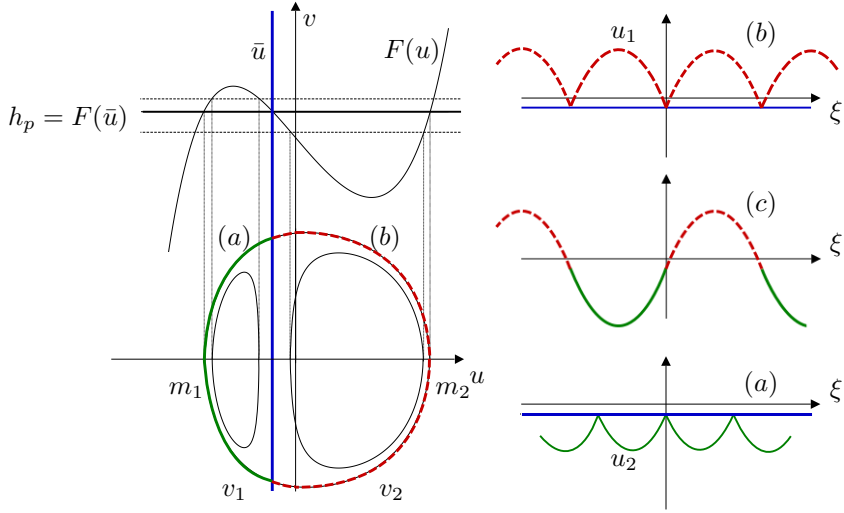

  \centering    
 \begin{lpic}[l(0mm),r(2mm),t(-5mm),b(0mm)]{peakedperiodicorbits_all(0.45)}
     \lbl[m]{88,135;$F(u)$}
     \lbl[m]{43,135;$\ub$}
     \lbl[r]{5,115;$h_p=F(\ub)$}
     \lbl[m]{22,35;$m_1$}
     \lbl[m]{100,35;$m_2$}
     
     \lbl[m]{107,39;$u$}
     \lbl[m]{58,144;$v$}
     
     \lbl[m]{212,117;$\xi$}
     \lbl[m]{212,78;$\xi$}
     \lbl[m]{212,35;$\xi$}         

     \lbl[m]{36,80;$(a)$}
     \lbl[m]{80,80;$(b)$}
     
     \lbl[m]{36,5;$v_1$}
     \lbl[m]{82,5;$v_2$}     

     \lbl[m]{150,140;$u_1$}
     \lbl[m]{150,14;$u_2$}
     
     \lbl[m]{190,140;$(b)$}
     \lbl[m]{190,90;$(c)$}
     \lbl[m]{190,37;$(a)$}
 \end{lpic}
 \caption{The choice of parameters $(c,s)\in  R_3\cup R_6$ gives rise to peaked periodic waves $(a)$ and $(b)$, as well as smooth periodic waves undulating about $\ub$ $(c)$. }
 \label{peakedperiod}
\end{figure}

\noindent Other than in Section \ref{Section:finitehomorbit} it is not possible to continuously extend the corresponding solutions by $\ub$ on the real line to obtain solitary waves with compact support, since here $\ub$ (or any other constant) does not satisfy equation \eqref{MAEODE}. However, we may continue the solutions periodically and obtain peaked periodic waves $u_{i}$, cf.~Figure \ref{peakedperiod}. The period of these waves is $2T_i$, and they have countably many points of discontinuity at the wave crests or troughs, $\xi=(2k+1)\,T_i$ where $k\in\Z$, $i=1,2$, respectively. When $T_1=T_2$, we obtain $\C^2$-periodic traveling waves $u_P$ with period $4T_1$, undulating about the flat surface at $\ub$. Indeed, a continuity argument guarantees the existence of parameters $(c,s) \in R_3\cup R_6$ such that $T_1=T_2$. In this case, the peaked periodic solutions $u_i$ obtained from \eqref{vi} may be glued  together at $\xi=(2k+1)\,T_1$, $k\in\Z$, cf.~Figure \ref{peakedperiod}$(c)$, which yields a smooth periodic solution $u_P$ undulating about $\ub$ defined on $\R$ by
\begin{align*}
     u_P(\xi)=\left\{
    \begin{array}{ll}
	    u_{1}(\xi),\;\text{ when } \xi \in \bigcup_{k=2m+1, m\in\Z} [(2k-1)T_1,(2k+1)T_1],\vspace{1em}\\
	    u_{2}(\xi),\;\text{ when } \xi \in \bigcup_{k=2m, m\in\Z} [(2k-1)T_1,(2k+1)T_1].
    \end{array}\right.
\end{align*}
This  continuation is $\C^2$ since $u_1(\xi)\rightarrow \ub$ as $ \xi\nearrow (2k+1)T_1$, and therefore  
\[
    \dot u_1 = v_1 = \pm\sqrt{2\frac{F(\ub-F(u)}{u-\ub} }\rightarrow \sqrt{-2F'(\ub)}  \text{ and } \ddot u_1\rightarrow 0 \text{ as }  \xi\nearrow (2k+1)T_1,
\]
and similarly $u_2(\xi)\rightarrow \ub$,  $\dot u_2(\xi)\rightarrow \sqrt{-2F'(\ub)}$ and $\ddot u_2(\xi)\rightarrow 0$ as $ \xi\searrow (2k+1) T_1$, for all $k\in\Z$. The same reasoning shows that the continuation is $\C^2$ at the lower bounds of the existence intervals. 


 \section{Properties of solitary traveling waves}
 \label{sect:Props}
The analysis in Section \ref{sect:Existence} shows that traveling wave solutions of \eqref{MAE} are symmetric with respect to the crest point, and that solitary waves tend (exponentially) to a constant on either side of their unique maximum or minimum. In the present Section, we will explore further properties regarding the shape of traveling waves. We determine how the wave amplitude, which is the positive difference between crest and trough, changes with respect to the wave speed. Furthermore, we prove that traveling waves are strictly monotone between crest and trough. 
 \subsection{Dependence of the amplitude on the wave speed}
 \label{subsect:ampl}
Our starting point is an algebraic expression for the change of $a=m-s$ with respect to $c$. 
\begin{Lemma} \label{Lem}
Let $F$ be the polynomial defined in \eqref{FMAE} and let $(s,m)$ be a solution of 
\begin{align}
    \left\{
    \begin{array}{ll}
       F'(s) = 0, \\
       F(s) - F(m)=0,
    \label{MaxSys}
    \end{array}\right.
\end{align} where $s\neq \ub$. Then, for $a = m - s \in \R$, we have that
\begin{equation}
 \label{adot}
  \partial_c\, a = \dfrac{-1/28}{F'(m)\,F''(s)}\,\Big((s^2-m^2)\,F''(s) + 2s\,F'(m)\Big).
\end{equation}
\end{Lemma}
\noindent In the following Proposition, we study the sign of  \eqref{adot}  to  determine the regions in the parameter set $\A$ which give rise to solitary waves whose amplitude $|a|$ is increasing or decreasing with respect to the wave speed  $c$, cf.~Figure \ref{amplituderegions}.
\begin{Proposition} We distinguish between the following cases:
    \begin{itemize}
      \item For $(c,s) \in R_1$ we obtain solitary waves of \emph{elevation} whose amplitude is strictly \emph{increasing} with $c$ in region (I), and \emph{decreasing} with $c$ in regions (II) and (III).
     \item For $(c,s) \in R_4$ we obtain solitary waves of \emph{depression} whose amplitude is strictly \emph{decreasing} with $c$ in region (IV) and \emph{increasing} with $c$ in region (V).
    \end{itemize}
 \label{Prop}
\end{Proposition}
\begin{figure}[h]
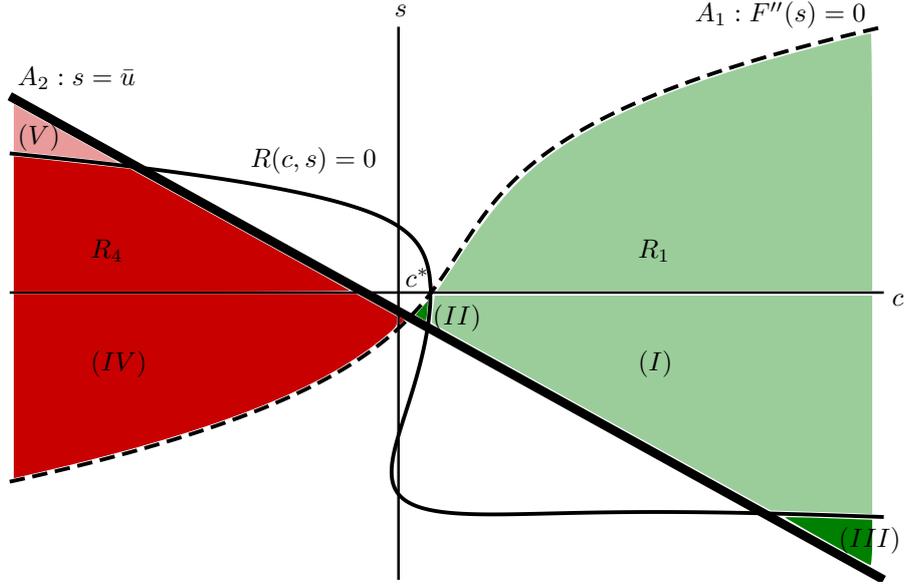

  \centering
    \begin{lpic}[l(0mm),r(0mm),t(-3mm),b(-5mm)]{regionsamp2(0.6)}
       \lbl[l]{140,50;$(I)$}
      \lbl[l]{184,11;$(III)$}
      \lbl[l]{95,60;$(II)$}
      \lbl[l]{89,69;$c^*$}
      \lbl[l]{20,50;$(IV)$}
      \lbl[l]{4,100;$(V)$}

      \lbl[m]{197,64;$c$}
      \lbl[m]{88,128;$s$}
      \lbl[l]{4,113;$A_2: s=\ub$} 
      \lbl[r]{190,127;$A_1: F''(s)=0$}

      \lbl[l]{55,95;$R(c,s)=0$}  
      
     \lbl[l]{140,75;$R_1$}
      \lbl[l]{20,75;$R_4$}
    \end{lpic}
   \caption{Choosing parameters in the lighter shaded regions  \textit{(I)} and \textit{(V)} we obtain solitary waves which \emph{increase} with the wave speed $c$, whereas solutions corresponding to parameters in the darker shaded regions \textit{(II), (III) and (IV)} \emph{decrease} with respect to $c$. 
   }
 \label{amplituderegions}
 \end{figure}
\begin{proof}[Proof of Proposition \ref{Prop}]
Denote $a = m-s$, so that the wave amplitude is given by $|a|$, and $a<0$ for waves of depression (in $R_4$) whereas  $a>0$ for waves of elevation (in $R_1$). 
 Observe that $F'(m) >0$ for all $(c,s) \in \A$ and recall that $F''(s) <0$ in $R_1$ and $F''(s) > 0$ in $R_4$. Therefore, and in view of \eqref{adot}, it suffices to study the sign of 
 \begin{equation*}
     (s^2-m^2)\,F''(s) + 2s\,F'(m),
 \end{equation*}
which in view of \eqref{Kcs} yields 
\begin{equation}
\label{numeradotKs}
    \frac{1}{14} \, \left( s-m \right) ^{2} \underbrace{ (6\,{m}^{2}s-4\,sm+12\,m{s}^{2}-2\,{s}^{2}+6\,{s}^{3}-1+c )}_{=:Q_{c,s}(m)}.
\end{equation}
Recall that $(s,m)$ solves \eqref{MaxSys} which reads
\begin{equation}
\label{MaxSysKs}
{\frac {1}{140}}\, ( s-m ) ^{2} \underbrace{( 6\,{m}^{3}-5\,{m}^
{2}+12\,{m}^{2}s+10\,m+18\,m{s}^{2}-10\,sm+5-5\,c-15\,{s}^{2}+24\,{s}^
{3}+20\,s ) }_{=:P_{c,s}(m)}.
\end{equation}
Since we are interested in solutions $s\neq m$, we study the system 
\begin{align}
\label{AmpSys}
\left \{
\begin{array}{l l}
     Q_{c,s}(m)&=0\\
     P_{c,s}(m)&=0
 \end{array}
 \right.
\end{align}
which has a solution if and only if $Q_{c,s}(m)$ and $P_{c,s}(m)$ have a common root, i.e.~if their resultant with respect to $m$, 
\begin{align}
\label{resultant}
  R(c,s) :=&\,Res(Q_{c,s}(m),P_{c,s}(m),m)\\
            =& \,31104\,{s}^{9}-10368\,{s}^{8}+32832\,{s}^{7}+( -15552\,c+39456) \,{s}^{6}\notag\\ 
 & \, + \left( -3816-864\,c \right) {s}^{5}+ \left( 23472-
3312\,c \right) {s}^{4}+24\, \left( c-1 \right)  \left( 108\,c-593
 \right) {s}^{3}\notag\\
&\,  +24\, \left( 33\,c+107 \right)  \left( c-1 \right) {s}
^{2} -690\, \left( c-1 \right) ^{2}s+36\, \left( c-1 \right) ^{3}\notag,
\end{align}
 is zero (see Appendix \ref{A:2} for a discussion on the involved curves). Hence, system \eqref{AmpSys} has a solution, i.e.~$\partial_c \, a=0$, only along the algebraic curve $R(c,s)=0$ meaning that within the regions in $\A$ separated by this curve, the sign of  $\partial_c \, a$ is constant. Hence, it suffices to pick one pair of parameters $(\bar{c},\bar{s})$ in each of these regions and compute the values of $m$ and $Q$ to determine the sign of $\partial_c \, a$ in view of expression \eqref{adot}. For example, let $\bar{s}_1=-0.1$ and $\bar{c}_1=1.5$ in $R_1$ then, computing  the corresponding $m_1$, we find that $Q_{\bar{c}_1,\bar{s}_1}(m_1)>0$. Therefore, since $F'(m)<0$ and $F''(s) <0$ in $R_1$, we obtain that $\partial_c \, a>0$. Hence, the amplitude $|a|$ of solitary wave solutions of \eqref{MAE} arising from parameters in the region denoted by $(I)$ in Figure \ref{amplituderegions} is increasing with respect to the wave speed $c$. To provide an example for waves of depression, pick $\bar{s}_4=-0.5$ and $\bar{c}_4=15$ in $R_4$, which yields $m_4$ such that $Q_{\bar{c}_4,\bar{s}_4}(m_4)<0$.  Therefore, since $F'(m)<0$ and $F''(s) >0$ in $R_4$, we obtain that $\partial_c \, a>0$. In view of the fact that $a<0$ in  $R_4$ this means that the amplitude $|a|$ of solitary waves with parameters in  region $(IV)$ is decreasing with respect to $c$.  The results for the other regions $(II),(III)$ and $(V)$ can be obtained in exactly the same way.  
\end{proof}

\begin{proof}[Proof of Lemma \ref{Lem}]
Consider $F$ in \eqref{FMAE}, regarding it as a polynomial in $u$ and $c$, and define 
\[
  f(u,c) := F(u)=  K\,u + \frac{1-c}{28}u^2 + \frac{1}{14}u^3 - \frac{1}{28}u^4 + \frac{3}{70}u^5.
\]
Then the first equation in \eqref{MaxSys} rewrites as 
\[
 f_u(s,c)=0,
\]
 where subscripts denote partial differentiation. Implicit differentiation with respect to $c$ of the last equation yields
\[
 f_{uu}(s,c)\,\dot s + f_{uc}(s,c) = F''(s)\,\dot s -\frac{1}{14} s = 0,
\]
where $\,\dot {} $ denotes differentiation with respect to $c$, and therefore
\begin{equation*}
 \label{sdot}
  \dot s = \frac{1}{14} \frac{s}{F''(s)}.
\end{equation*}
The second equation in \eqref{MaxSys} reads
\[
  f(s,c) - f(m,c) = 0,
\]
which upon implicit differentiation yields
\[
 f_u(s,c)\,\dot s + f_c(s,c) - f_u(m,c)\,\dot m - f_c(m,c)=0
\]
so
\begin{equation*}
 \label{mdot}
  \dot m = \dfrac{f_c(s,c) - f_c(m,c)}{f_u(m,c)} = \dfrac{-1/28}{F'(m)}\,(s^2 - m^2).
\end{equation*}
Since $\dot a = \dot m - \dot s$, this proves \eqref{adot}. 
\end{proof}

\subsection{Monotonicity}
\label{subsect:monot}
We show that the profile $u$ of a solitary traveling wave solution of \eqref{MAE} is monotone from the undisturbed water level $s$ to its maximum or minimum $m$, with precisely one inflection point on either side of the wave crest or trough. To this end, consider the right hand side of \eqref{v},
\begin{equation*}
    v=\sqrt{\vphantom{2\,\dfrac{F(s) - F(u)}{u-\ub}}\smash{\underbrace{2\,\dfrac{F(s) - F(u)}{u-\ub}}_{=:D(u)}}}. \vspace{1em}
\end{equation*}
We claim that, as a  function of $u$, this expression has a unique critical point between $s$ and $m$. Since the square root is monotonous, it suffices to consider the number of critical points of the discriminant $D(u)$ of this expression. By construction, $D(u)$ has a critical point at $s$. We will show that there exists precisely one more critical point in $(s,m)$, which corresponds to the unique inflection point of the wave profile $u$ between its trough and crest. To this end, we study $D'(u)$ and prove that it has exactly one real root to the right of $s$. For simplicity, we will give the proof only in the case $u>\ub$, the other case can be proven in exactly the same way. Indeed, consider the numerator of $D'(u)$, which in view of \eqref{Kcs} yields 
\begin{equation}
\label{eq_monot_Iscu}
    (s-u)\,I_{c,s}(u),
\end{equation}
where $I_{c,s}(u)$ is a fourth-order polynomial in $u$ whose coefficients depend polynomially on the parameters $s$ and $c$. In Appendix \ref{A:3} we prove that the number of roots of $I_{c,s}(u)$ in $(s,\infty)$ remains constant if we vary the parameters $(c,s) \in \mathcal{A}$  (cf.~also Lemma 3.6 in \cite{GarGasGia13}). Therefore, it suffices to  evaluate the polynomial at any point $(\bar c,\bar s) \in \A$ and to deduce that the resulting univariate polynomial has a unique real root in $(\sb,\infty)$ employing Sturm's method (cf.\cite{Stoer1980}). We conclude that $D'(u)$ has a unique real root to the right of $s$ for parameters in $\A$, which proves the claim.

\section{Traveling waves of the Camassa-Holm equation}
\label{sect_CH}
We would like to point out that the method to prove existence of traveling waves put forward in Section \ref{sect:Existence} is applicable to a wide class of nonlinear dispersive evolution equations. As an example, we apply our approach to the Camassa-Holm equation, which is usually written in the form
\begin{equation}
 \label{CH}
  u_t + 2\kappa\,u_x -u_{txx} + 3\, u\,u_x  = 2\,u_x u_{xx} + u\,u_{xxx},
\end{equation}
for $x \in \R$, $t>0$ and $\kappa\in\R$. For traveling waves $u(x,t)=u(x-c\,t)$, equation \eqref{CH} takes the form
\begin{equation*}
 u''(u-c) + \frac{(u')^2}{2} + K + (c-2\kappa)\,u - \frac{3}{2}u^2 = 0,
\end{equation*} 
where $K$ is a constant of integration. If instead of $u$ we study the translate 
\[
  w=u-c,
\]
the previous equation reads
\begin{equation}
\label{CHODE}
 w''\,w + \frac{1}{2}(w')^2 + F'(w)= 0,
\end{equation} 
where 
\begin{equation}
 \label{FCH}
  F(w)=  A\,w - B\,w^2 - \frac{1}{2}w^3,
\end{equation}
with constants $A=K-2\kappa c - \frac{1}{2}c^2$ and $B=c+\kappa $. Now, equation \eqref{CHODE} is of the form \eqref{MAEODE} and we may prove existence of traveling wave solutions as in Section \ref{sect:Existence}:

\begin{Theorem}
 \label{Prop:CH}
 There exist solitary and periodic traveling wave solutions of the Camassa-Holm equation \eqref{CH} for every $c$, $K$ and $\kappa$ satisfying
  \begin{equation}
    \label{boundsCH}
      -\frac{2}{3}\, B^2 < A <-\frac{1}{2}\,B^2,
  \end{equation}
  with constants $A$ and $B$ defined as above. All solitary waves are symmetric with respect to their unique  maximum/minimum and tend exponentially to a constant at infinity. Periodic waves exist also for $A>-\frac{1}{2}B^2$. 
\end{Theorem}
\begin{proof}
Bounded orbits in the phase plane associated to \eqref{CHODE} give rise to traveling waves of \eqref{CH} in view of Proposition \ref{Prop:orbits_waves} as before. To work out conditions for the existence of homoclinic and periodic orbits we proceed along the lines of the proof of Proposition \ref{Lem:Conditions_orbits}. Bounded orbits exist as long as the local extrema of $F$ are distinct, i.e. when the discriminant of $F'$, discrim$(F',w)= 6A + 4B^2$, is greater than zero. This yields the lower bound in \eqref{boundsCH}. To guarantee the existence of homoclinic orbits we have to ensure that there is one saddle point and one center point in the phase plane. To this end, we study the relation $F(s)=F(0)$, where $s$ is the solution of $F'(s)=0$ with $F''(s)<0$ for $s>0$ and $F''(s)>0$ for $s<0$, which yields the curve $A=-\frac{1}{2} B^2$ marking the upper bound in  \eqref{boundsCH}.
\end{proof}
\begin{Proposition}
There exist peaked continuous solitary traveling wave solutions of the Camassa-Holm equation \eqref{CH} for $c$, $K$ and $\kappa$ satisfying $A=-\frac{1}{2}\,B^2$. 
\end{Proposition}
\begin{proof}
For parameters satisfying $A=-\frac{1}{2}\,B^2$ we obtain homoclinic orbits which give rise to continuous solitary traveling wave solutions of \eqref{CH} with a peaked crest. Indeed, for this choice of parameters $h_s=F(s)=0$, so from relation \eqref{v} we get 
\begin{equation*}
  v= \pm (B-u) \;\text{ and }\; v' = \pm 1.
\end{equation*}
Hence, $v\rightarrow \pm B$ when $u\rightarrow 0$, so there is a discontinuity of $v$ at $(0,\pm B)$, the crestpoint of the solitary wave solution. However, it is straightforward to check that such a solution still satisfies the equation \eqref{CHODE} in this point.
\end{proof}

\begin{Remark}
 Solitary traveling waves decaying to the flat surface at zero are known to exist when $c>2\kappa$, cf.~\cite{Constantin2002}. This result is reflected in the condition $A> -\frac{2}{3}\, B^2$ for $K=0$. Moreover, we recover the fact that peaked solitons (peakons) exist when $\kappa =0$, cf.~\cite{Camassa}, from the relation $A=-\frac{1}{2}\,B^2$ for $K=0$.
\end{Remark}

\appendix
\section{Algebraic curves}
\label{A:}
We want to provide some remarks on the algebraic curves involved in our analysis and show that our figures display correctly the graphs of the corresponding expressions. We only exemplify the procedure by proving some selected cases. First,  recall a result on the number of roots of polynomials (cf.~\cite{GarGasGia13}), which we will use repeatedly:
\begin{Lemma}
\label{L_roots}
Consider an interval $\Omega\subset\R$ and a family of real polynomials whose coefficients depend continuously on a real parameter $b$,
\[
   G_b(x)= g_{n}(b)x^{n} + g_{n-1}(b)x^{n-1}+\dots+ g_1(b)x+g_0(b).
\]
 Suppose there exists an open interval $I\subset\R$ such that:
\begin{enumerate}[(i)]
 \item There is some $b_0\in I$, such that $G_{b_0}(x)$ has exactly $k$ simple roots on $\Omega$.
 \item For all $b\in I$, the discriminant of $G_b$ with respect to $x$ is different from zero.
 \item For all $b\in I$, $g_n(b)\neq 0$.
\end{enumerate}
Then for all $b\in I$, $G_b(x)$ has exactly $k$ simple roots on $\Omega$. Moreover, if  $\Omega=\Omega_b :=((c(b),\infty)\subset\R$ for some continuous function $c(b)$ the same result holds if we add the hypothesis: 
\begin{enumerate}
\item[(iv)]  For all $b\in I$, $G_b(c(b))\neq 0$.
\end{enumerate}
\end{Lemma}

\noindent The intuition behind this result is as follows: In view of the hypotheses $(i)-(iv)$, the roots of $G_b(x)$ depend continuously on $b$. Assumptions $(iii)$ and $(iv)$ impede possible bifurcations of roots from infinity or from the boundary of $\Omega$ when varying $b\in I$. Moreover, assumption $(ii)$ prevents the appearance of double real roots in the interior of $\Omega$. Therefore, the number of roots of $G_b(x)$ is constant in $\Omega$ when the parameter $b \in I$ varies, and hence $G_b(x)$ has $k$ simple roots for all $b\in I$ in view of assumption $(i)$.

\subsection{On the curves in  Section \ref{subsect:CondExistence}}
\label{A:1}
We have to ensure that the algebraic curve $M(c,s)=0$ in \eqref{boundADis}  has no real roots,  and that  $N(c,s)$ in \eqref{boundAdF} has a unique root for every choice of parameters $(c,s)$. Furthermore, we want to show that the curves $N(c,s)=0$ and $A_1$, $A_2$  of \eqref{boundsA} all intersect in precisely one point. We have
\begin{align*}
 M(c,s) &=243\,{c}^{2}+ \left( -900\,s-778+540\,{s}^{2}-1080\,{s}^{3} \right) c \\
           & + 823+1284\,s +480\,{s}^{2}+480\,{s}^{3}+2700\,{s}^{4}-1296\,{s}^{5}+1296
\,{s}^{6},
\end{align*}
which we regard as a polynomial in $c$ with parameter $s$. To show that it has no real roots, we check that the conditions of Lemma \ref{L_roots} are satisfied for $k=0$. Assumption $(iii)$ holds, since the coefficient of the highest order term is constant. Computing the discriminant of $M(c,s)$ with respect to $c$ yields
\begin{equation*}
    -16\, \left( 18\,{s}^{2}-6\,s+23 \right) ^{3},
\end{equation*}
for which it is straightforward to prove that it has no zeros in $\R$. To check assumption $(i)$ we choose, for example, $s=-1$ which gives 
\[
    M(c,-1)= 243\,{c}^{2}+1742\,c+4831>0.
\]
In view of Lemma \ref{L_roots}, we find that $M(c,s)$ is strictly positive for all $c,s\in\R$. Next we focus on 
\[
    N(c,s) = 3\,{c}^{3}+ \left( -42\,s+37 \right) {c}^{2}+ \left( -476\,s+588\,{s}^
{2}+3397 \right) c+6076\,{s}^{2}-8232\,{s}^{3}-8666\,s-2125.
\]
It is straightforward to show that assumptions $(ii)$ and $(iii)$ hold. Regarding $(i)$, we choose for example $s=1$ and find that 
\[
    N(c,1)=(3c-11)\,(c^2+2c+1177)
\]
which clearly has a unique root. Hence, for each $s\in\R$ there exists precisely one $c\in\R$ such that $N(c,s)=0$. To prove that the curves $A_1$, $A_2$ and $N(c,s)=0$ all intersect in precisely one point, observe that the $A_i$ are linear in $c$ and it is therefore easy to check that they intersect at a point $(c^*,s^*)$ where $s^*$ is the root of the cubic polynomial $P_*:=-2-20\,s+6\,{s}^{2}-12\,{s}^{3}$, and $c^*$ the corresponding value on the curve $A_2$. To see that $N(c,s)=0$ also intersects in that point, compute the resultant of $A_2$ and $N$ with respect to $c$ and find that the resulting polynomial is just a factor of $P_I$ which proves the claim. 

\subsection{On the curves in  Section \ref{subsect:ampl}}
\label{A:2}
In this subsection, we want to discuss the curves in \eqref{numeradotKs}, \eqref{MaxSysKs} and \eqref{resultant}. We show that the curve $R(c,s)=0$ intersects the curve $A_2$ three times whereas it intersects $A_1$ only once, which yields the different regions depicted in Figure \ref{amplituderegions}. To prove the latter result, we solve both $A_1$ and $A_2$ for $c$ and plug the resulting expressions into $R(c,s)$. We obtain univariate polynomials in $s$ for which it is straightforward to show that they have three roots (two negative and one positive) and one root (in zero), respectively.
To show that the polynomial expressions defining the involved curves yield a unique root for each $(c,s)\in\R^2$, we again employ Lemma \ref{L_roots} and check that the assumptions are satisfied. We exemplify the procedure by showing that the result is true for the curve $P_{c,s}(m)=0$, where 
\begin{align*}
    P_{c,s}(m) & = -6\,{m}^{3}+ ( 5-12\,s )\, {m}^{2}+ ( -18\,{s}^{2}+10\,s-10) \,m
                        -24\,{s}^{3} +15\,{s}^{2}-20\,s -5+5\,c.
\end{align*}
 Note that we are now dealing with a polynomial which depends on two parameters. Assumption $(iii)$ holds in view of the fact that the coefficient of the highest order term is constant, and it is straightforward to check $(i)$ choosing parameters $(\cb,\sb)$ which yield that $P_{\cb,\sb}(m)$ has precisely one root. To prove assumption $(ii)$, we compute the discriminant of $P_{c,s}(m)$ with respect to $m$ and obtain
\begin{align*}
    \text{Dis}&(P_{c,s}(m),m)  = 
        -24300\,{c}^{2}+ ( 120600\,s-75600\,{s}^{2}+73100+151200\,{s}^{3}) c\\
        & -70300-96200\,{s}^{2}-163600\,s-504000\,{s}^{4}+259200\,{s}^
{5}+2400\,{s}^{3}-259200\,{s}^{6}.
\end{align*}
To show that Discrim($P_{c,s}(m)$) is different from zero, we repeat the scheme for this polynomial in $c$ with parameter $s$ and ensure again that the assumptions of  Lemma \ref{L_roots} hold with $k=0$.  

\subsection{On the curves in  Section \ref{subsect:monot}}
\label{A:3}
The goal of this subsection is to prove that for the polynomial 
\begin{align*}
    I_{c,s}(u)  = & -168\,{u}^{4}+ \left( 90-15\,c-168\,s \right) {u}^{3} 
                  + \left( 10\,c-130-168\,{s}^{2}+90\,s-15\,sc \right) {u}^{2}\\
                &+(-50+10\,sc+20\,c -130\,s+90\,{s}^{2}-168\,{s}^{3}-15\,{s}^{2}c)u\\
                & -5-50\,s+90\,{s}^{3}+5\,{c}^{2}+20\,sc-130\,{s}^{2}-168\,{s}^{4}
                    +10\,{s}^{2}c-15\,{s}^{3}c,
\end{align*}
obtained in  \eqref{eq_monot_Iscu}, the number of roots do not change if we vary the parameters $(c,s)$ in the admissible region $\A$. We will again employ Lemma \ref{L_roots} above for $I=\A$ and $\Omega_{c,s}=(s,\infty)$. Indeed, assumption $(iii)$ holds since the highest coefficient of $I_{c,s}(u)$ is constant. To check that $(iv)$ is satisfied, we evaluate the polynomial at $u=s$ and find that
\begin{equation*}
  I_{c,s}(s)= F''(s)\cdot(s-\ub).
\end{equation*}
These factors are exactly the relations which bound the admissible parameter region  $\mathcal{A}$, and hence they do not vanish in the interior of $\A$. Note, however, that solitary waves with compact support arise from a choice of parameters $(c,s)$ on the curve $A_2=\{s-\ub=0\}$, so we need a separate argument in that case which will be carried out below. Next we study the discriminant of $I_{c,s}(u)$ with respect to $u$, 
\begin{equation}
\label{eq_monot_Discrim}
 \text{ Dis}(I_{c,s}(u),u)= D_1(c,s)\cdot D_2(c,s),
\end{equation}
and claim that the algebraic curves corresponding to the zeros of these factors lie outside of $\A$. To see this, observe that the curves $A_1$, $A_2$, $\{D_1(c,s)=0\}$ and $\{D_2(c,s)=0\}$ intersect precisely once in the point $(c^*,s^*)$. Then, we choose some $c_1 < c^*$ and find that $D_1(c_1,s)>A_1(c_1,s)>A_2(c_1,s)>D_2(c_1,s)$, whereas for any $c_2> c^*$ we obtain the reverse order. Hence, the discriminant of $I_{c,s}(u)$ does not vanish in $\A$, which proves the claim. Therefore, the assumptions of Lemma \ref{L_roots} hold, and we find that the number of roots of  $I_{c,s}(u)$ is constant in the interior of $\A$.\\

\noindent We now provide a separate but similar argument which asserts that this result holds  also for parameters on the curve $A_2$. Indeed, for $(c,s)$ on $A_2$, i.e.~when $s=\ub$, we find that 
\begin{equation*}
    D'(u)=\tilde I_{c}(u), 
\end{equation*}
where $\tilde I_{c}(u)$ is a cubic polynomial in $u$ whose coefficients depend polynomially on $c$. Along the lines of the above proof we argue that $\tilde I_{c}(u)$ has a unique real root in $(\ub,\infty)$. Indeed, no bifurcations of roots occur at infinity, and  evaluating $\tilde I_{c}(u)$ at the boundary $u=\ub$ yields a cubic polynomial in $c$ which vanishes only in $c^*\notin \A$. Using Sturm's method (cf.\cite{Stoer1980}) we show that the discriminant of $\tilde I_{c}(u)$ with respect to $u$ does not vanish. Hence, the number of roots is constant, and choosing any $\bar{c}$ we find that $\tilde I_{\bar{c}}(u)$ has a unique real root in $(\ub_{\bar{c}},\infty)$.\\

\noindent We conclude with a discussion of the polynomials in \eqref{eq_monot_Discrim},
\begin{align*}
    D_1 =& -32928\,{s}^{3}+ ( 1764\,c+22344) {s}^{2}
              - ( 84\,{c}^{2}+1148\,c+28504) s
              +3\,{c}^{3}+44\,{c}^{2}+7919\,c-5842
\end{align*}
and
\begin{align*}
    D_2 = & \, 30375\,{c}^{5} 
               +( 67500\,{s}^{2}-135000\,{s}^{3}+93100+567900\,s) {c}^{4} \\
             &+ ( 2083200\,{s}^{2}-3518100\,{s}^{4}+408900\,s+162000\,{s}^{6}         
               +1280400\,{s}^{3}-162000\,{s}^{5}+880703) {c}^{3}\\
            & + ( -368064\,s-4730400\,{s}^{6}-8347536\,{s}^{2}+5443200\,{s}^{7}
               +6777000\,{s}^{4}-11014128\,{s}^{3}\\
            & -25691040\,{s}^{5} -3605574) {c}^{2} 
               + ( -44997120\,{s}^{7}+4400084-15110352\,s+23678784\,{s}^{5}\\
            & -9163584\,{s}^{6}+60963840\,{s}^{8}-22971024\,{s}^{2}
               -17525952\,{s}^{3}-58261680\,{s}^{4}) c+227598336\,{s}^{9}\\
            & -138184704\,{s}^{8}+31667136\,s+301625856\,{s}^{7}+71568192\,{s}^{3}
               +152350848\,{s}^{6}+1062232\\
            & +54393984\,{s}^{5}+49720800\,{s}^{2}+187454304\,{s}^{4}.
\end{align*}
We will only discuss the latter curve and employ Lemma \ref{L_roots} again  for $I \times \Omega=\A$. Note that assumption $(iii)$ holds in view of the fact that the coefficient of the highest order term is constant. Computing the discriminant with respect to $c$ yields 
\begin{align*}
   \text{Dis}(D_2,c) = & \,\alpha_0\,
    ( 11664000\,{s}^{12}-23328000\,{s}^{11}+367804800\,{s}^{10}-487728000\,{s}^{9}\\
          &  +3390049800\,{s}^{8} -2253805200 \,{s}^{7}+4960871884\,{s}^{6}
              +2160459976\,{s}^{5}\\
         &+1280057526\,{s}^{4}
         + 4059678628\,{s}^{3} +1729573411\,{s}^{2}+1328288220\,s+695918709)  \\
         & \times ( 84672\,{s}^{4}+22512\,{s}^{3}+76402\,{s}^{2}+58822\,s+16767) ^{3}\\
        & \times (\underbrace{16767+29606\,s-12083\,{s}^{2}-4040\,{s}^{3}+20160\,{s}^{4}}
                              _{=:\tilde D(s)})^{2},
\end{align*}
where $\alpha_0>0$ is a real constant. It is straightforward to see that the first two factors of the above expression have no real roots, whereas the last factor $\tilde D(s)$ vanishes for two values $s_1$ and $s_2$, meaning that $D_2(c,s_i)$ may have multiple roots in that case. To ensure that assumption $(ii)$ holds, we have to prove that these values do not lie in $\A$. To this end, we use Sturm's method to derive rational upper and lower bounds for $s_i$ such that $s_1\in[\underline{s}_1,\bar{s}_1]$ and  $s_2\in[\underline{s}_2,\bar{s}_2]$, to bound the curves $A_i$ from above and below by rational constants. For example, we find that $s_1 \in [-\frac{131}{128},-\frac{65}{64}]=:I_1$. We claim that all values of $D_2$ in the strip defined by $I_1$ lie outside of $\A$. To this end we compute a bound for $A_2$,
\[
        M_1:=\min_{s\in I_1}{\{c \in \R : A(c,s)=0\}}= \frac{423}{32}\in \mathbb{Q},
\]
and construct a rational univariate polynomial 
\begin{align*}
    \tilde D_2(c) = & \,30375\,{c}^{5}-{\frac {8897979325}{32768}}\,{c}^{4}-{\frac {
602652229378808443}{274877906944}}\,{c}^{3}+{\frac {
3196448247290763459}{137438953472}}\,{c}^{2}\\
&+{\frac {
2641868472255829530863}{70368744177664}}\,c-{\frac {
17460388511693021202337}{35184372088832}},
\end{align*}
using the upper and lower bounds of the interval $I_1$ such that $D_2(c,s) < \tilde D_2(c)$. Then, it is fairly straightforward to see that $D_2(c,s) -M_1 < \tilde D_2(c) -M_1<0$, which proves the claim. Repeating this procedure with the root $s_2$ and the other bounding curves $A_i$ shows that in the region $\A$ all involved curves are displayed correctly.

\subsection*{Acknowledgements}
The first author is partially supported by a MCYT- FEDER grant number MTM2008-03437 and by a CIRIT grant number 2009SGR 410. The second author is  supported by the FWF project J3452 ''Dynamical Systems Methods in Hydrodynamics`` of the Austrian Science Fund.

\end{document}